\documentclass[12pt]{article}
\usepackage{amssymb}
\usepackage{mathptmx}      
\usepackage{amsmath,amsthm}

\usepackage{amsfonts}
\usepackage{latexsym}
\usepackage{amssymb}
\newtheorem{theorem}{Theorem}
\newtheorem{proposition}{Proposition}
\newtheorem{lemma}{Lemma}
\newtheorem{definition}{Definition}
\newtheorem{remark}{Remark}
\newtheorem{example}{Example}
\newcommand{\R}{\mathbb R}
\newcommand{\RR}{\overline{\mathbb R}}
\newcommand{\E}{\mathbb E}
\newcommand{\ld}[3]{#2^{(#1)}_{-} #3}
\newcommand{\lsubd}[2]{\partial^{(#1)}_{-} #2}
\newcommand{\subd}[2]{\partial #1 #2}
\newcommand{\sld}[2]{#1^{\pr\pr}_{-} #2}

\newcommand{\pr}{\prime}
\newcommand{\eps}{\varepsilon}
\newcommand{\f}{f : \E\to\R\cup\{+\infty\}}
\newcommand{\scalpr}[2]{\langle#1,#2\rangle}
\newcommand{\norm}[1]{\Vert#1\Vert}
\newcommand{\dom}[1]{{\rm dom}\,#1} 
\newcommand{\dini}[1]{#1^{\pr}_D}
\newcommand{\secdini}[1]{#1^{\pr\pr}_D}
\newcommand{\ph}{\varphi}

\title{Second-Order Optimality Conditions in Cone-Constrained Vector Optimization with Arbitrary Nondifferentiable Functions}

\author{Vsevolod I. Ivanov}

\begin{document}
\maketitle

\begin{abstract}
In this paper, we introduce a new second-order directional derivative and a second-order subdifferential of Hadamard type for an arbitrary nondifferentiable function. We derive several second-order optimality conditions for a local and a global minimum and an isolated local minimum of second-order in unconstrained optimization. In particular, we obtain two types results with generalized convex functions.
We also compare our conditions with the results of the recently published paper [Bedna\v r\'ik, D., Pastor, K.: On second-order conditions in unconstrained optimization. Math. Program. Ser A, {\bf 113}, 283--291 (2008)] and a lot of other works, published in high level journals, and prove that they are particular cases of our necessary and sufficient ones. We prove that the necessary optimality conditions concern more functions than the lower Dini directional derivative, even the optimality conditions with the last  derivative can be applied to a function, which does not belong to  some special class. At last, we apply our optimality criteria for unconstrained problems to derive necessary and sufficient optimality conditions for the cone-constrained vector problems.
\end{abstract}

Key words:\quad Nonsmooth optimization$\cdot\;$  Optimality conditions$\cdot\;$ Second-order directional derivative of Hadamard type  $\cdot\;$ Second-order isolated minimizers $\cdot\;$ Strict local minimizers of second-order $\cdot\;$  Generalized convex functions

MSC 2010:\quad 90C46 $\cdot\;$ 49J52 $\cdot\;$ 90C26 $\cdot\;$ 26B25

\section{Introduction}
\label{s1}
In our opinion the main aim of nondifferentiable optimization is to extend some results to as more as possible general classes of functions. The task to obtain optimality conditions in unconstrained optimization is old. There are first- and second-order necessary and sufficient conditions which concern several classes of functions (C$^{1,1}$, C$^1$, locally Lipschitz, lower semicontinuous and so on) in terms of various generalized derivatives. For all of them we should check that the function belongs to some special class, which is not easy sometimes. There are a lot of second-order generalized directional derivatives, whose necessary and sufficient conditions for optimality have similar proofs (see, for example, the references \cite{aus84,bp04,pas08,bt85,cha94,cha87,com91,cc90,ggr06,husn84,hua05,hua94,jl98,roc88,roc89,yan96,yan99,yj92}). This fact motivated us to find another derivative such that these conditions follow from the second-order ones in term of it, the second-order necessary conditions and the sufficient ones in unconstrained optimization are satisfied for arbitrary nondifferentiable function and the derivative coincides with the second-order Fr\'echet directional derivative in the case when the last one exists. They can be applied in nonlinear programming, for example, for solving the problem with penalty functions or reduce the problem to convex composite.  

In this paper, we introduce a new second-order generalized directional derivative. We obtain necessary and sufficient conditions for a local minimum and isolated local minimum of a function in terms of this derivative. In the conditions, we suppose that the function is arbitrary proper extended. Additionally, we derive second-order conditions, which are necessary and sufficient for a given point to be a global minimizer. They concern a new class of invex functions. We prove necessary and sufficient first-order conditions for a given point to be an isolated minimizer of order two of a strongly pseudoconvex function. Our generalized derivatives have the advantage that the proofs of the optimality conditions are simple. On the other hand, they are satisfied for arbitrary function, not necessarily with locally Lipschitz gradient, or continuously differentiable, or locally Lipschitz, or continuous, even not necessarily semicontinuous. We also compare our necessary and sufficient conditions with the respective ones in the references \cite{aus84,bp04,bt85,cha94,cha87,com91,cc90,ggr06,husn84,hua05,hua94,jl98,roc88,roc89,yan96,yan99,yj92}. We prove that the conditions in all these works are simple consequences of our necessary and sufficient conditions. On the other hand, the proofs given there are not so short sometimes. For example, the main result in the recently published in the journal Mathematical Programming paper \cite{pas08} is to extend the conditions for an isolated local minimum  in unconstrained optimization to $l$-stable functions. This is a class of functions, whose lower Dini directional derivatives satisfy a property, which is analogous to Lipschitz one. They include all C$^{1,1}$ functions. We prove that the main Theorem 6 in this paper follows from our Theorem \ref{th2} when the function is $l$-stable at the candidate for minimizer and continuous near it. Therefore, it is not necessary to guess and check if the function is $l$-stable. We also compare the necessary conditions in terms of Hadamard and Dini derivatives. We prove that our conditions are preferable. They concern more functions. 

At last, we obtain necessary and sufficient conditions for optimality in cone-constrained vector optimization. In particular, our results are satisfied for problems with inequality and equality constraints.

\section{A new second-order directional derivative and subdifferential of Hadamard type}
We suppose that $\E$ is a real finite-dimensional Euclidean space. Denote by $\R$ the set of reals and $\RR=\R\cup\{-\infty\}\cup\{+\infty\}$.
Let us consider the following second-order directional derivative at the point $x\in\E$ in direction $u\in\E$ of a given function $f$, defined in the space $\E$, which was introduced in \cite{gin02}:
\[
f^{[2]}_G(x;u):=\liminf_{t\downarrow 0,u^\pr\to u}\,2t^{-2}[f(x+tu^\pr)-f^{[0]}_G(x;u)-t f^{[1]}(x;u)],
\]
where $f^{[0]}_G(x;u):=\liminf_{t\downarrow 0,u^\pr\to u}\,f(x+tu^\pr)$ and 
\[
f^{[1]}_G(x;u):=\liminf_{t\downarrow 0,u^\pr\to u}\,t^{-1}[f(x+tu^\pr)-f^{[0]}_G(x;u)].
\]
  The optimality conditions for unconstrained problems were derived for arbitrary nondifferentiable function. Suppose that the function is twice  Fr\'echet differentiable. Then
$f^{[0]}_G(x;u)=f(x)$, $f^{[1]}_G(x;u)=\nabla f(x)(u)$ and 
\begin{equation}\label{IG}
f^{[2]}_G(x;u):=\liminf_{t\downarrow 0,u^\pr\to u}\,2t^{-2}[f(x+tu^\pr)-f(x)-t\nabla f(x)(u)].
\end{equation}
 We obtain by second-order Taylor's formula with a reminder in the form of Peano \cite{atf87} that
\[
f(x+t u^\pr)=f(x)+\nabla f(x)(t u^\pr)+\frac{1}{2}\nabla^2 f(x)(tu^\pr)(tu^\pr)+o(t^2),
\]
 where $\lim_{t\downarrow 0} o(t^2)/t^2=0$. Therefore,
\[
f^{[2]}_G(x;u)=\nabla^2 f(x)(u)(u)+\liminf_{t\downarrow 0,u^\pr\to u}\, 2t^{-1}[\nabla f(x)(u^\pr-u)].
\]
It follows from here that if $\nabla f(x)\ne 0$, then $f^{[2]}_G(x;u)\ne\nabla^2 f(x)(u)(u)$. Really, we have $f^{[2]}_G(x;u)=-\infty$ for every direction $u\in\E$. Hence, $f^{[2]}_G(x;u)$ is not exactly a derivative. 
Our task is to define a second-order directional derivative such that the necessary conditions and the sufficient ones in unconstrained optimization are satisfied for arbitrary nondifferentiable function and the derivative coincides with the second-order Fr\'echet directional derivative provided that the last one exists. 
How can we do this keeping the convergence $u^\pr\to u$? A possible decision is to denote $v=(u^\pr-u)/t$ and take $v\to 0$. Thus, we obtain the derivative
\[
f^{!!}(x;u):=\liminf_{t\downarrow 0,v\to 0}\,2t^{-2}[f(x+tu+t^2 v)-f(x)-t\nabla f(x)(u)].
\]
This derivative and also higher-order ones were studied in the work \cite{bas01}. In the present paper, we develop another idea to replace in Equation (\ref{IG}) in the expression $\nabla f(x)(u)$ the variable $u$ by $u^\pr$. Thus, we obtain the derivative
\[
\ld 2 f(x;u)=\liminf_{t\downarrow 0,u^\pr\to u}\,2t^{-2}[f(x+tu^\pr)-f(x)-t\nabla f(x)(u^\pr)].
\]

We introduce a new second-order derivative and a second-order subdifferential, which are based on the presented observation. Let $X$ and $Y$ be two linear spaces and $L(X,Y)$ be the space of all continuous linear operators from $X$ to $Y$. Then denote by $L^1(\E)$ the space $L(\E,\R)$, by $L^2(\E)$ the space $L(\E,L^1(\E))$.   Consider a proper extended real function $\f$, that is a function, which never takes the value $-\infty$ and at least one value is finite. The domain of a proper extended real function is the set:
${\rm dom}\; f:=\{x\in\E| f(x)<+\infty\}.$

\begin{definition}
The lower Hadamard directional derivative of a function $\f$
at a point $x\in\dom f$ in direction $u\in\E$ is defined as follows:
\[
\ld 1 f(x;u)=\liminf_{t\downarrow 0,u^\pr\to u}\,
t^{-1}[f(x+t u^\pr)-f(x)].
\]
\end{definition}

It follows from this definition that, if $\ld 1 f(x;u)$ is finite, then the 
direction $u$ belongs to the Bouligand tangent cone of the domain of the function $f$.

\begin{definition}
Recall that the lower Hadamard subdifferential of the function $\f$ at the point
$x\in\dom f$ is defined  by the following relation:
\[
\lsubd 1 f(x)=\{ x^*\in L^1(\E)\mid x^*(u)\le\ld 1 f (x;u)\quad\textrm{for all directions}\quad u\in\E\}.
\]
\end{definition}

We introduce the following definitions:

\begin{definition}\label{df-2had}
Let $\f$ be an arbitrary proper extended real function.
Suppose that $x^*_1$ is a fixed element from the lower Hadamard subdifferential
$\lsubd 1 f (x)$ at the point $x\in\dom f$. Then the lower
second-order derivative of Hadamard type of $f$ at $x\in\dom f$ in direction $u\in\E$ is
defined as follows:
\[
\ld 2 f(x;x^*_1;u)=\liminf_{t\downarrow 0,u^\pr\to u}\,
2t^{-2}[f(x+t u^\pr)-f(x)-tx^*_1(u^\pr)].
\]

\end{definition}

\begin{definition}
Let $\f$ be an arbitrary proper extended real function. Suppose that $x\in\dom f$, $x^*_1\in\lsubd 1 f(x)$.
The lower second-order Hadamard subdifferential of the function $\f$ at the point
$x\in\dom f$ is defined  by the following relation:
\[
\lsubd 2 f(x;x^*_1)=\{ x^*\in L^2(\E)\mid x^*(u)(u)\le\ld 2 f ( x;x^*_1;u)\;\textrm{for all directions}\; u\in\E\}.
\]
\end{definition}

The next claim follows from the given above discussion .

\begin{proposition}
Let the function $f:\E\to\R$ be twice Fr\'echet differentiable at the point $x$ and Fr\'echet differentiable on some neighborhood of $x$ with first- and second-order Fr\'echet derivatives $\nabla f$ and $\nabla^2 f$. Then
\[
\ld 1 f(x;u)=\nabla f(x)(u);\quad\lsubd 1 f(x)=\{\nabla f(x)\},\quad
\]
\[
\ld 2 f(x;\nabla f(x);u)=\nabla^2 f(x)(u)(u)
\]
\end{proposition}

\section{Optimality conditions for a local minimum and isolated local minimum of second-order}

In this section, we obtain optimality conditions for unconstrained problems in terms of the introduced second-order derivative.

\begin{theorem}\label{th1}
Let $\bar x\in\dom f $ be a local minimizer of the function $f$. Then
\begin{equation}\label{15}
\ld 1 f(\bar x;u)\ge 0,\quad\ld 2 f(\bar x;0;u)\ge 0,\quad\textrm{ for all }u\in\E.
\end{equation}
\end{theorem}
\begin{proof}
Since $\bar x$ is a local minimizer, then it follows from the definition of the lower Hadamard directional derivative that
there exists a neighborhood $N\ni\bar x$ with $f(x)\ge f(\bar x)$ for all $x\in N$. Let $u\in\E$ be an arbitrary chosen direction. Then $f(\bar x+t u^\pr)\ge f(\bar x)$ for all sufficiently small positive numbers $t$ and for all directions $u^\pr$, which are  sufficiently close to $u$. It follows from here that $\ld 1 f (\bar x;u)\ge 0$. Therefore $0\in\lsubd 1 f(\bar x)$, because $u\in\E$ is an arbitrary direction.

By the definition of the second-order lower derivative,
using that $0\in\lsubd 1 f(\bar x)$ we obtain that $\ld 2 f(\bar x;0;u)$ is well defined and
\[
\ld 2 f (\bar x;0;u)=\liminf_{t\downarrow 0,u^\pr\to u,}\, 2\, t^{-2}
[f(\bar x+t u^\pr)-f(\bar x)]\ge 0
\]
for all directions $u\in\E$. 
\qed\end{proof}

\begin{remark}
Condition (\ref{15}) is equivalent to the following one:
\begin{equation}\label{14}
0\in\lsubd 1 f(\bar x)\quad\textrm{and}\quad  0\in\lsubd 2 f(\bar x;0).
\end{equation}
\end{remark}

The following definition is well known. 

\begin{definition}
A point $\bar x\in\dom f$ is called an isolated local minimizer of second-order for the function $\f$ iff there exist a neighborhood $N$ of $\bar x$ and a constant $C>0$ with
\begin{equation}\label{1}
f(x)>f(\bar x)+C\norm{x-\bar x}^2,\quad\forall  x\in N, x\ne \bar x.
\end{equation}
\end{definition}

\begin{theorem}\label{th2}
Let be given a proper extended real function $\f$ and $\bar x\in\dom f$. Then the following claims are equivalent:

a) $\bar x$ is an isolated local minimizer of second-order;

b) the following conditions holds for all $u\in\E$:
\begin{equation}\label{6}
\ld 1 f(\bar x;u)\ge 0\quad\textrm{and}\quad \ld 2 f(\bar x;0;u)>0, u\ne 0; 
\end{equation}

c) the following conditions
\begin{equation}\label{4}
\ld 1 f(\bar x;u)\ge 0,\quad\forall u\in\E
\end{equation}
and 
\begin{equation}\label{5}
u\ne 0,\;\ld 1 f(\bar x;u)=0,\quad\Rightarrow\quad
\ld 2 f(\bar x;0;u)>0,.
\end{equation}
are satisfied.
\end{theorem}
\begin{proof}
It is obvious that the implication b) $\Rightarrow$ c) holds. We prove a) $\Rightarrow$ b).  Let $\bar x$ be an isolated local minimizer of second-order. We prove that Conditions (\ref{6}) hold.
Suppose that $u\in\E$ is arbitrary chosen. It follows from Inequality (\ref{1}) that there exist numbers $\delta>0$, $\varepsilon>0$ and $C>0$ with
\begin{equation}\label{3}
f(\bar x+tu^\pr)\ge f(\bar x)+Ct^2\norm{u^\pr}^2
\end{equation}
for all  $t\in (0,\delta)$ and every $u^\pr$ such that $\norm{u^\pr-u}<\varepsilon$. Therefore 
\begin{equation}\label{2}
\ld 1 f (\bar x;u)=\liminf_{t\downarrow 0,u^\pr\to u,}\,t^{-1}[f(\bar x+t u^\pr)-f(\bar x)]\ge\liminf_{t\downarrow 0,u^\pr\to u,}\, Ct\norm{u^\pr}^2=0.
\end{equation}
According to Inequality (\ref{2}) we have   $0\in\lsubd 1 f(\bar x)$. 
It follows from (\ref{3}) that
\[
\ld 2 f(\bar x;0;u)=\liminf_{t\downarrow 0,u^\pr\to u,}\,2t^{-2}[f(\bar x+t u^\pr)-f(\bar x)] \\
\ge\liminf_{t\downarrow 0,u^\pr\to u,}\, 2\, C\norm{u^\pr}^{2}=2\, C\norm{u}^2>0
\]
for all directions $u$ such that $u\ne 0$.

We prove c) $\Rightarrow$ a). Suppose that conditions {\rm (\ref{4})} and {\rm (\ref{5})}  hold. We prove that $\bar x$ is an isolated local minimizer of second-order.
Assume the contrary that $\bar x$ is not an isolated minimizer.
Therefore, for every sequence $\{\varepsilon_k\}_{k=1}^\infty$ of positive numbers
converging to zero, there exists a sequence $\{x_k\}$ with $x_k\in\dom f$ such that
\begin{equation}\label{60}
\norm{x_k-\bar x}\le \varepsilon_k,\quad
f(x_k)< f(\bar x)+\varepsilon_k\norm{x_k-\bar x}^2,
\end{equation}

It follows from (\ref{60}) that $x_k\to\bar x$. Denote $t_k=\norm{x_k-\bar x}$,
$d_k=(x_k-\bar x)/t_k$. Passing to a subsequence, we may suppose
that $d_k\to  d$ where $\norm{d}=1$. It follows from here that
\[
\ld 1 f (\bar x;d)\le\liminf_{k\to\infty}\, t^{-1}_k[f(\bar x+t_k d_k)-f(\bar x)]
\le\liminf_{k\to\infty}\varepsilon_k t_k=0.
\]
It follows from (\ref{4}) that $0\in\subd f(\bar x)$ and $\ld 1 f (\bar x;d)=0$. We have
\[
\ld 2 f (\bar x;0;d)\le\liminf_{k\to\infty}\,2 t^{-2}_k[f(\bar x+t_k d_k)-f(\bar x)]\le\liminf_{k\to 0}\, 2\varepsilon_k=0 
\]
which is contrary to (\ref{5}).
\end{proof}

\section{Conditions for a global minimum of a second-order invex function}

The following question arises: Which is the largest class of functions such that the necessary conditions  from Theorem \ref{th1} become sufficient for a global minimum. 
Recently Ivanov \cite{optimization} introduced a new class of Fr\'echet differentiable functions called second-order invex ones in terms of the classical second-order directional derivative. They extend the so called invex ones and obey the following property: A Fr\'echet differentiable function is second-order invex if and only if each second-order stationary point is a global minimizer. We generalize this notion to arbitrary nondifferentiable functions in terms of the lower Hadamard directional derivatives of second-order. 

We recall the definition of an invex function \cite{han81}. We apply the lower Hadamard directional derivative here.

\begin{definition}
 A proper extended real function $\f$ is called invex in terms of the lower Hadamard directional derivative iff there exists a map
$\eta: \E\times \E\to\E$  such that the following inequality holds for all $x\in\dom f $, $y\in\E$: 
\begin{equation}\label{13}
f(y)-f(x)\ge\ld 1 f(x;\eta(x,y)). 
\end{equation}
\end{definition}

We introduce the following two definitions:
\begin{definition}
We call a function $\f$ second-order  invex (for short, $2$-invex) in terms of the lower Hadamard derivatives  iff for every $\bar x\in\dom f$, $x\in\E$  with $0\in\lsubd 1 f(\bar x)$
there are $\eta_1$, $\eta_2$, which depend on $\bar x$ and $x$ such that the  following inequality holds 
\begin{equation}\label{8}
f(x)-f(\bar x)\ge\ld 1 f(\bar x;\eta_1(\bar x,x))+\ld 2 f(\bar x;0;\eta_2(\bar x,x)).
\end{equation}
\end{definition}

\begin{definition}\label{df-2st}
Let $\f$ be a given proper extended real function. We call   every point $\bar x\in\dom f$ such that 
\[
\ld 1 f(\bar x;u)\ge 0,\; \ld 2 f(\bar x;0;u)\ge 0,\quad\forall u\in\E
\]
second-order stationary  (for short, $2$-stationary point).
\end{definition}


\begin{theorem}
Let  $\f$ be a proper extended real function.  Then $f$ is second-order invex if and only if each second-order stationary point $\bar x\in\dom f$ is a global minimizer of $f$.
\end{theorem}
\begin{proof}
Suppose that $f$ is second-order invex. If the function has no stationary points, then obviously every second-order stationary point is a global minimizer. Suppose that the function has at least one second-order stationary point $\bar x$, that is a point satisfying Definition \ref{df-2st}. We prove that it  is a global minimizer.
Let  $x$ be an arbitrary point from $\E$. It follows from second-order invexity that there exist $\eta_i(\bar x,x)$, $i=1,2$ such that Condition (\ref{8}) is satisfied.
Since $\bar x$ is a second-order stationary point, then 
\[
\ld 1 f(\bar x;\eta_1(\bar x,x))\ge 0,\;\ld 2 f(\bar x;0;\eta_2(\bar x,x))\ge 0.
\]
It follows from (\ref{8}) that $f(x)\ge f(\bar x)$. Therefore $\bar x$ is a global minimizer.

Conversely, suppose that every second-order stationary point is a global minimizer. We prove that $f$ is second-order invex. 
Assume the contrary.  Hence, there exists a pair $(\bar x,x)\in\dom f\times \E$ such that $0\in\lsubd 1 f(\bar x)$ and   the following inequality holds
\begin{equation}\label{10}
f(x)-f(\bar x)<\ld 1 f(\bar x;y)+\ld 2 f(\bar x;0;z),\quad\forall y\in\E,\;\forall z\in\E.
\end{equation}

First, we prove that $f(x)<f(\bar x)$. Let us choose in (\ref{10}) $y=0$, $z=0$. We have 
\[
\ld 1 f(\bar x;0)\le\liminf_{t\downarrow 0}\,t^{-1}(f(\bar x+t.0)-f(\bar x))=0,
\]
\[\begin{array}{c}
\ld 2 f(\bar x;0;0)= \liminf_{t\downarrow 0,u^\pr\to 0}\,2t^{-2}[f(\bar x+t.u^\pr)-f(\bar x)]\le 0.
\end{array}\] 
 It follows from (\ref{10}) that $f(x)<f(\bar x)$.

Second, we prove that  
\begin{equation}\label{11}
\ld 1 f(\bar x;u)\ge 0,\quad\forall u\in\E.
\end{equation} 
Suppose the contrary that there exists at least one point $v\in\E$ with $\ld 1 f(\bar x;v)<0$. The lower Hadamard directional derivative is positively homogeneous with respect to the direction, that is 
\[
\ld 1 f(\bar x;\tau u)=\tau\ld 1 f(\bar x;u)\quad\forall \bar x\in\dom f,\;\forall u\in\E,\;\forall \tau\in(0,+\infty).
\]
Then inequality (\ref{10}) cannot be satisfied when $y=tv$ with $t$ being sufficiently large positive number and $z=0$, because $\ld 2 f(\bar x;0;0)\le 0$ and $f(x)-f(\bar x)>-\infty$. Therefore, $\ld 1 f(\bar x;u)\ge 0$ for all $u\in\E$.

Third, we prove that 
\begin{equation}\label{12}
\ld 2 f(\bar x;0;u)\ge 0,\quad\forall u\in\E.
\end{equation}
Suppose the contrary that there exists $v\in\E$ with $\ld 2 f(\bar x;0;v)<0$. Then (\ref{10}) cannot be satisfied for all points $y=0$, $z=tv$, where $t$ is a sufficiently large positive number,  because the lower Hadamard directional derivative of second order is positively homogeneous of second degree with respect to the direction. Indeed, $\ld 1 f(\bar x;0)=0$, thanks to (\ref{11}), and $f(x)-f(\bar x)>-\infty$.

The following is the last part of the proof.   It follows from (\ref{11}) and (\ref{12}) that $\bar x$ is a second-order stationary point. According to the hypothesis $\bar x$ is a global minimizer, which is impossible, because $f(x)<f(\bar x)$. 
\end{proof}

In the next claim we show that the class of second-order invex functions includes all invex ones in terms of the lower Hadamard directional derivative  functions.

\begin{proposition}
Let $\f$ be an invex function. Then $f$ is second-order invex.
\end{proposition}
\begin{proof}
It follows from Equation (\ref{8})  that $f$ is second-order invex keeping the same map $\eta_1$ and taking $\eta_2=0$, because
$\ld 2 f(\bar x;0;0)\le 0$.
\end{proof}

The converse claim is not satisfied. There are a lot of second-order invex functions, which are not invex. The following example is extremely simple.
\begin{example}
Consider the function $f:\R^2\to\R$ defined by 
\[
f(x_1,x_2)=-x_1^2-x_2^2.
\]
We have $\ld 1 f(x;u)=-2x_1 u_1-2x_2 u_2$, where $u=(u_1,u_2)$ is a direction. Its only stationary point is $\bar x=(0,0)$. This point is not a global minimizer. Therefore, the function is not invex. We have $(0,0)\in\lsubd 1 f(\bar x)$ and $\ld 2 f(\bar x;0;u)=-2u_1^2-2u_2^2$. It follows from here that $f$ has no second-order stationary points. Hence, every second-order stationary point is a global minimizer, and the function is second-order invex.
\end{example}

\section{Strongly pseudoconvex functions and second-order isolated minimizers}
\label{s6}
Strongly pseudoconvex functions were introduced by Diewert, Avriel and Zang \cite{d-a-z81}. 
Their definition assumes additionally strict pseudoconvexity. It was proved by Hadjisavvas and Schaible \cite{had93} that in the differentiable case, strict pseudoconvexity of the function is superfluous; in other words each function, which satisfies the next definition is strictly pseudoconvex.

\begin{definition}[\cite{had93}]\label{def2} 
Let $S$ be an open convex subset of $\E$. A Fr\'echet differentiable function $f:S\to\R$ is said to be strongly pseudoconvex iff, for all $x\in S$, $u\in\E$ such that $\norm{u}=1$ and $\nabla f(x)(u)=0$, there exist positive numbers $\delta$ and $\alpha$ with $x+\delta u\in S$ and
\[
f(x+t u)\ge f(x)+\alpha t^2,\quad 0\le t<\delta.
\]
\end{definition}

In this section, we derive optimality conditions for an isolated minimum of order two of a function, which satisfies the strong pseudoconvexity at some point only.
We consider the definition of a strongly pseudoconvex function in terms of the lower Dini directional derivative. 

\begin{definition}\label{df-spc}
We call a function $f:\E\to\R$ strongly pseudoconvex at the point $x\in\dom f$ iff $\dini f(x;u)=0$, $u\in\E$, $\norm{u}=1$ implies that there exist positive
numbers $\delta$ and $\alpha$ with
\[
f(x+t u)\ge f(x)+\alpha t^2,\quad\forall t\in(0,\delta).
\]
\end{definition}



\begin{definition}
The first- and second-order lower Dini directional derivatives of a function $\f$ at the point $x\in\dom f$ in direction $u\in\E$ are defined as follows:
\[
\dini f(x;u)=\liminf_{t\downarrow 0}\, t^{-1}[f(x+tu)-f(x)].
\]
\[
\secdini f(x;u)=\liminf_{t\downarrow 0}\, 2t^{-2}[f(x+tu)-f(x)-t\dini f(x,u)].
\]
\end{definition}

The following notion extends the Lipschitz continuity of the gradient.

\begin{definition}[\cite{pas08}]
A function $f:\E\to\R$ is called $\ell$-stable at the point $x\in\E$ iff there exist a neigbourhood $U$ of $x$ and a constant $K>0$ such that 
\[
|\dini f(y;u)-\dini f(x;u)|\le K\,\norm{y-x}\,\norm{u},\quad\forall y\in U,\;\forall u\in\E.
\]
\end{definition}

\begin{proposition}[\cite{pas08}]\label{pr1}
Let the function $f:\E\to\R$ be continuous on some neighborhood of $x\in\E$ and $\ell$-stable at $x$. Then $f$ is strictly differentiable at $x$, hence Fr\'echet differentiable at $x$.
\end{proposition}

The following mean-value theorem is due to Diewert \cite{die81}.
\begin{lemma}\label{mv}
Let $\ph:[a,\,b]\to\R$ be a lower semicontinuous function of one real variable.
Then there exists $\xi$, $a<\xi\le b$, such that 
\[
\ph(a)-\ph(b)\le \dini \ph(\xi; a-b).
\]
\end{lemma}

\begin{lemma}\label{lema4} 
Let $f:\E\to\R$ be radially lower semicontinuous on some neighborhood of $x\in\E$ and $l$-stable at $x$. Suppose that $\dini f(x;u)=0$ for all $u\in\E$.  Then the following limit exists and it equals $0$:
\begin{equation}\label{40}
\lim_{t\downarrow 0,h^\pr\to h}\; [f(x+t h^\pr)-f(x+th)]/t^2=0.
\end{equation}
\end{lemma}
\begin{proof}
By Diewert's mean-value theorem, there exists $\theta\in [0,1)$ such that
\[
f(x+t h^\pr)-f(x+th)\le \dini f(x+th+t\theta(h^\pr-h);t(h^\pr-h))
\]
\[
=t[\dini f(x+th+t\theta(h^\pr-h);h^\pr-h)-\dini f(x;h^\pr-h)].
\]
Since $f$ is $l$-stable there exists $K>0$ with
\[
\dini f(x+th+t\theta(h^\pr-h);h^\pr-h)-\dini f(x;h^\pr-h)\le tK\norm{h+\theta(h^\pr-h)}\norm{h^\pr-h}.
\]
Therefore 
\begin{equation}\label{30}
\limsup_{t\downarrow 0,h\pr\to h}\; [f(x+t h^\pr)-f(x+th)]/t^2\le 0.
\end{equation}

On the other hand, by mean-value theorem, there exists $\tau\in [0,1)$ such that
\[
f(x+th)-f(x+th^\pr)\le \dini f(x+th^\pr+t\tau(h-h^\pr);t(h-h^\pr))
\]
\[
=t[\dini f(x+th^\pr+t\tau(h-h^\pr);h-h^\pr)-\dini f(x;h-h^\pr)].
\]
Since $f$ is $l$-stable there exists $K>0$ with
\[
\dini f(x+th^\pr+t\tau(h-h^\pr);h-h^\pr)-\dini f(x;h-h^\pr)\le t K\norm{h^\pr+\tau(h-h^\pr)}\norm{h^\pr-h}.
\]
Therefore 
\begin{equation}\label{31}
\liminf_{t\downarrow 0,h^\pr\to h}\; [f(x+t h^\pr)-f(x+th)]/t^2\ge 0.
\end{equation}
Then the lemma follows from (\ref{30}) and (\ref{31}).
\end{proof}

In the next theorem, we derive necessary and sufficient conditions for an isolated local minimum of second-order of a strongly pseudoconvex function at some point $\bar x$:

\begin{theorem}\label{th3}
Let the function $f:\E\to\R$ be continuous on some neighborhood of $\bar x\in\E$ and $\ell$-stable at $\bar x$.
 Suppose that  $f$ is strongly pseudoconvex only at the point $\bar x$. 
Then $\bar x$ is an isolated local minimizer of second-order if and only if $\nabla f(\bar x)=0$.
\end{theorem}
\begin{proof}
Let $\bar x$ be an isolated local minimizer of second-order.
We conclude from Proposition \ref{pr1} that $\nabla f(\bar x)$ exists. Then it is obvious that  $\nabla f(\bar x)=0$.

We prove the converse claim. Suppose that $\nabla f(\bar x)=0$, but
$\bar x$ is not an isolated local minimizer of second-order. Therefore,
for every sequence $\{\varepsilon_k\}_{k=1}^\infty$ of positive numbers
converging to zero, there exists a sequence $\{x_k\}$, $x_k\in\dom f$ such that inequalities (\ref{60}) hold.
It follows from here that $x_k\to \bar x$. Denote $t_k=\norm{x_k-\bar x}$,
$d_k=(x_k-\bar x)/t_k$.  Passing to a subsequence, we may suppose
that the sequence $\{d_k\}_{k=1}^\infty$ is convergent and $d_k\to  d$,  where $\norm{d}=1$.
Therefore
\begin{equation}\label{25}
\liminf_{k\to\infty}\, t^{-2}_k
[f(\bar x+t_k d_k)-f(\bar x)]=\liminf_{k\to \infty}\, t^{-2}_k[f(x_k)-f(\bar x)]\le\lim_{k\to \infty}\, \varepsilon_k=0.
\end{equation}
We have
\[
f(\bar x+t_k d)-f(\bar x)=[f(\bar x+t_k d)-f(\bar x+t_kd_k)]+[f(\bar x+t_kd_k)-f(\bar x)].
\]
It follows from Lemma \ref{lema4} and (\ref{25}) that
\[
\liminf_{k\to\infty}\, t^{-2}_k [f(\bar x+t_k d)-f(\bar x)]\le 0.
\]
On the other hand, according to Definition \ref{df-spc}
\[
f(\bar x+t_k d)\ge f(\bar x)+\alpha t_k^2
\]
for all sufficiently large $k$. Hence,
\[
\liminf_{k\to\infty}\, t^{-2}_k
[f(\bar x+t_k d)-f(\bar x)]\ge\liminf_{k\to\infty}\,\alpha=\alpha>0,
\]
which is a contradiction.
\end{proof}

The following example shows that Theorem \ref{th3} is not true for functions, which are not $\ell$-stable:

\begin{example}\label{ex1}
Consider the function 
\[
f=|\,x_2-\sqrt[3]{x_1^4}\,|^{\, 3/2}.
\]
Of course, the point $\bar x=(0,0)$ is a local and global minimizer, but it is not an isolated local minimizer of order two. Even it is not a strict local minimizer, because $f(x)=0$ for all $x=(x_1,x_2)$ over the curve $x_2=x_1^{4/3}$. We have $\nabla f(\bar x)=(0,0)$.
Simple calculations show that this function is strongly pseudoconvex in the sense of Definition \ref{df-spc} at $\bar x$.  Let $v=(v_1,v_2)$ be an arbitrary vector, whose norm is $1$. If $v_2>0$ or $v_2<0$, then 
\[
\lim_{t\downarrow 0}\,[f(\bar x+tv)-f(\bar x)]/t^2=\lim_{t\downarrow 0}\, f(tv)/t^2=+\infty.
\]
 If $v_2=0$, then $v_1=\pm 1$ and 
\[
\lim_{t\downarrow 0}\,[f(\bar x+tv)-f(\bar x)]/t^2=\lim_{t\downarrow 0}\, f(tv)/t^2=1.
\]
Therefore, for every $v\in\R^2$ there exists $\delta>0$ and $C>0$ such that 
\[
f(tv)>Ct^2 \quad\textrm{for all}\quad t\in(0,\delta).
\]
The sufficient conditions of Theorem \ref{th3} are not satisfied, because $f$ is not $\ell$-stable at $\bar x$. Indeed,
if we take
$x^\pr_k=(0,k^{-1})$, then $\nabla f(x_k)=(0,3/2k^{-1/2})$ and there do not exist $K>0$ such that 
\[
\norm{\nabla f(x^\pr_k)-\nabla f(\bar x)}\le K\norm{x^{\pr}_k-\bar x}
\]
for all sufficiently large integers $k$. We have
\[
\lim_{k\to+\infty}\norm{\nabla f(x^{\pr}_k)}\, / \,
\norm{x^{\pr}_k}=+\infty.
\]
\end{example}

We adopt the definition of a strongly pseudoconvex functions to proper extended real functions in terms of the lower Hadamard derivative and derive optimality conditions for such functions:
\begin{definition}
We call a proper extended real function  $\f$ 
strongly pseudoconvex at $x\in\dom f$ iff, for every $d\in\E$ such that
$\norm{d}=1$ and $\ld 1 f(\bar x;d)=0$,
there exist positive numbers $\eps$, $\delta$, and $\alpha$ such that
\begin{equation}\label{45}
f(\bar x+td^\pr)\ge f(\bar x)+\alpha t^2,
\end{equation}
for all $t\in\R$ and $d^\pr\in\E$, $\|d^\prime\|=1$ with $0\le t<\delta$, $\|d^\pr-d\|<\varepsilon$.
\end{definition}

\begin{theorem}\label{th5}
Let $\f$ be a strongly pseudoconvex function at $\bar x$.
Then $\bar x$ is an isolated local minimizer of second-order if and only if and only if $0\in\lsubd 1 f(\bar x)$. 
\end{theorem}
\begin{proof}
It follows from Theorem \ref{th1} that $\bar x$ is an isolated local minimizer
of second-order implies 
\[
0\in\lsubd 1 f(\bar x)\quad\textrm{and}\quad  \ld 2 f (\bar x;0;u)\ge 0,\;\forall u\in\E.
\] 

Suppose that $0\in\lsubd 1 f(\bar x)$. We prove that $\bar x$ is an isolated minimizer of second-order.
Assume the contrary.  Therefore, for every sequence $\{\varepsilon_k\}_{k=1}^\infty$ of positive numbers
converging to zero, there exists a sequence $\{x_k\}$, $x_k\in\dom f$ such that inequalities (\ref{60}) hold and $x_k\to\bar x$. Denote $t_k=\norm{x_k-\bar x}$,
$d_k=(x_k-\bar x)/t_k$. Without loss of generality  we may suppose
that $d_k\to  d$,  where $\norm{d}=1$. Therefore
\[
\ld 1 f (\bar x;d)\le\liminf_{k\to\infty }\, t^{-1}_k[f(\bar x+t_k d_k)-f(\bar x)]
\le\liminf_{k\to\infty}\,\varepsilon_k t_k=0.
\]
By the assumption $0\in\lsubd 1 f(\bar x)$, we conclude from here  that 
$\ld 1 f(\bar x;d)=0$. On the other hand, by strong pseudoconvexity we obtain that
there exist positive numbers $\eps$, $\delta$, and $\alpha$ such that Inequality (\ref{45}) is satisfied.
for all $t\in\R$ and $d^\pr\in\E$ with $\|d^\prime\|=1$, $0\le t<\delta$, $\norm{d^\pr-d}<\varepsilon$. Hence,
\[
\ld 2 f (\bar x;0;d)=\liminf_{t\downarrow 0,d^\pr\to d} 2t^{-2} [f(\bar x+t d^\pr)-f(\bar x)]\ge\liminf_ {t\downarrow 0,d^\pr\to d}2\alpha=2\alpha>0,
\]
which is a contradiction.
\end{proof}

\section{Necessary and sufficient conditions for cone-constrained vector problems}

Consider the multiobjective nonlinear programming problem

\bigskip
$C$-minimize $f(x)$ subject to $g(x)\in -K$,\hfill (P)
\bigskip

\noindent
where $f:X\to\R^n$ and $g:X\to\R^m$ are given vector-valued functions, defined on some open set $X\subset\R^s$, $C$ and $K$ are given closed convex cones with a vertex at the origin of the respective space. We suppose that $C\in\R^n$  has nonempty interior ${\rm int}(C)$.  Denote by $S$ the feasible set, that is 
\[
S:=\{x\in X\mid g(x)\in -K\}.
\]

\begin{definition}
A feasible point $\bar x$ is called a weak local minimizer iff there exists a neighborhood $N\ni\bar x$ such that there is no another feasible point $x\in S\cap N$ with $f(x)\in f(\bar x)-{\rm int}(C)$.
\end{definition}

Denote by $a\cdot b$ the scalar product between the vectors $a\in\R^n$ and $b\in\R^n$. Denote the positive polar cone of $C$ by $C^*$, that is 
\[
C^*:=\{\lambda\in R^n\mid\lambda\cdot x\ge 0\textrm{ for all }x\in C\},
\]
and the positive polar cone of $C^*$ by $C^{**}$.  

We begin with some preliminary lemmas.

\begin{lemma}[\cite{ggt04}]\label{C**}
Let $C$ be a nonempty closed convex cone in the $n$-dimensional space $\R^n$, whose vertex is the origin. Then $C^{**}=C$.
\end{lemma}

\begin{lemma}\label{lemaVP1}
Let $C$ be a closed convex cone, and $x\notin C$. Then there exists $\lambda\in C^*$ such that $\lambda\cdot x<0$.
\end{lemma}
\begin{proof}
Assume the contrary that $\lambda\cdot x\ge 0$ for all $\lambda\in C^*$. It follows from here that $x\in C^{**}$. On the other hand, by Lemma \ref{C**}, we have $C=C^{**}$, which contradicts the hypothesis $x\notin C$.
\end{proof}


\begin{lemma}\label{lemaVP3}
Let $C$ be a closed convex cone and $x\in C$. Then  $x\in {\rm int}(C)$ if and only if $\lambda\cdot x>0$ for all $\lambda\in C^*$ with $\lambda\ne 0$.
\end{lemma}
\begin{proof} 
Let $x\in {\rm int}(C)$. We prove that $\lambda\cdot x>0$ for all $\lambda\in C^*$ with $\lambda\ne 0$.
Suppose the contrary that there exists $\lambda\in C^*$ with $\lambda\cdot x\le 0$, $\lambda\ne 0$. It follows from the definition of the positive polar cone that $\lambda\cdot x=0$. There exists a number $\delta>0$ such that $x-\delta\lambda\in{\rm int}(C)$, because $\lambda\in\R^n$. By $\lambda\in C^*$ we have
$\lambda\cdot(x-\delta\lambda)\ge 0$. We obtain from here that $\lambda=0$, which is a contradiction.

Let us prove the converse claim. Suppose that $\lambda\cdot x>0$ for all $\lambda\in C^*$ with $\lambda\ne 0$, but $x\notin {\rm int}(C)$.  It follows from $x\notin{\rm int}(C)$ that there exists an infinite sequence $x_k$, converging to $x$, such that $x_k\notin C$. It follows from Lemma \ref{lemaVP1} that there exists $\lambda_k\in C^*$ such that $\lambda_k\cdot x_k<0$. We conclude from here that $\lambda_k\ne 0$. Without loss of generality, we suppose that
$\norm{\lambda_k}=1$ for all positive integers $k$. Passing to a subsequence we could suppose that $\lambda_k$ converges to some point $\lambda_0\ne 0$. Taking the limits when $k\to +\infty$, we obtain that $\lambda_0\cdot x\le 0$. Since the polar cone is always closed, we conclude that $\lambda_0\in C^*$. We conclude from here that $\lambda_0\cdot x>0$, which is a contradiction.
\end{proof}

Suppose that $\bar x\in S$ is a weak local minimizer for the problem (P). Let us consider the function
\[
F(x):=\max\{\lambda\cdot[f(x)-f(\bar x)]+\mu\cdot g(x)\mid (\lambda,\mu)\in\Lambda\},
\]
where $\Lambda:=\{(\lambda,\mu)\mid\lambda\in C^*,\;\mu\in K^*,\;\sum_{i=1}^n\lambda_i^2+\sum_{j=1}^m\mu_j^2=1\}$.

\begin{lemma}\label{F}
Suppose that $\bar x\in S$ is a weak local minimizer for the problem (P). Then there exists a neighborhood $N\ni\bar x$ such that $F(x)\geq F(\bar x)=0$ for all $x\in N$.
\end{lemma}
\begin{proof}
Using that $\bar x$ is weakly efficient, we conclude that  there exists a neighborhood $N\ni\bar x$ with the property that there is no another feasible point $x\in S\cap N$ with $f(x)\in f(\bar x)-{\rm int}(C)$. 
Let $x$ be an arbitrary point from $N$. Consider two cases:

Let $x\in N\cap S$. Then $f(x)-f(\bar x)\notin-{\rm int}(C)$. By Lemma \ref{lemaVP3} there exists $\lambda\in C^*$, $\lambda\ne 0$ such that $\lambda\cdot [f(x)-f(\bar x)]\geq 0$. Let us take $\mu=0$. Without loss of generality we suppose that $(\lambda,\mu)\in\Lambda$. Then it follows from the definition of the function $F$ that $F(x)\geq\lambda\cdot [f(x)-f(\bar x)]\geq 0$.

Let $x\in N\setminus S$. It follows from here that $g(x)\notin -K$. According to Lemma \ref{lemaVP1} there exists $\mu\in K^*$ such that $\mu\cdot g(x)>0$, $\mu\ne 0$. Let us take $\lambda=0$. Then without loss of generality we can suppose that $(\lambda,\mu)\in\Lambda$. According to the definition of the function $F$, we have $F(x)\geq\mu\cdot g(x)>0$.

Then taking into account both cases and that $F(\bar x)=0$, we have that $F(x)\geq F(\bar x)$ for every point $x\in N$.
\end{proof}

\begin{theorem}\label{thVP1}
Suppose that $\bar x\in S$ is a weak local minimizer for the problem (P). Then
\[
\ld 1 F(\bar x;u)\ge 0,\quad\ld 2 F(\bar x;0;u)\ge 0\quad\textrm{ for all }u\in\E.
\]
\end{theorem}
\begin{proof}
The theorem directly follows from Theorem \ref{th1} and Lemma \ref{F}.
\end{proof}

For the problem 
\[
R^n_+{\rm -minimize}\quad f(x)\quad \textrm{subject to}\quad x\in S,
\]
where $R^n_+$ is the positive orthant in $\R^n$  was introduced the following notion about isolated local minimizers under the name strict local minimizer:

\begin{definition}[\cite{jim02}]\label{jimenez}
A point $\bar x\in S$ is called a strict local minimizer of order $k$ iff there exists a constant $A>0$ and a neighborhood $N$ of $\bar x$ such that
\begin{equation}\label{34}
(f(x)+R^n_+)\cap B(f(\bar x),A\|x-\bar x\|^k)=\emptyset,\quad\forall x\in S\cap N\setminus\{\bar x\}.
\end{equation}
\end{definition}

Really, this definition is equivalent to the following one:
\begin{definition}\label{isolmin}
A point $\bar x\in S$ is called an isolated local minimizer of order $k$ iff there exists a constant $A>0$ and a neighborhood $N$ of $\bar x$ such that for every $x\in S\cap N\setminus\{\bar x\}$ there exists $i\in\{1,2,\dots,n\}$, which depend on $x$, with
\begin{equation}\label{29}
f_i(x)>f_i(\bar x)+A\|x-\bar x\|^k,
\end{equation}
where $\|a\|:=\sqrt{\sum_{i=1}^n a_i^2}$.
\end{definition}

\begin{proposition}
Definitions \ref{jimenez} and \ref{isolmin} are equivalent.
\end{proposition}
\begin{proof}
Suppose that $\bar x$ is an isolated local minimizer of order $k$ in the sense of Definition \ref{isolmin}, but (\ref{34}) does not hold. Therefore there exists $a=(a_1,a_2,\dots,a_n)\in\R^n$ such that
$a_i\geq f_i(x)$ for every $i$ and $\|a-f(\bar x)\|\leq A\|x-\bar x\|^k$. It follows from here that $f_i(x)\leq f_i(\bar x)+A\|x-\bar x\|^k$, which is a contradiction to (\ref{29}).

Suppose that $\bar x$ is a strict local minimizer of order $k$ in the sense of Definition \ref{jimenez}, but (\ref{29}) does not hold. Therefore, for every constant $A>0$ and every neighborhood $N\ni\bar x$ there exists  $x\in S\cap N\setminus\{\bar x\}$ such that
\[
f_i(x)\leq f_i(\bar x)+A\|x-\bar x\|^k,\quad\forall i\in\{1,2,\dots,n\}.
\]
It follows from here that
\[
(f(x)+R^n_+)\cap B(f(\bar x),nA\|x-\bar x\|^k)\ne\emptyset,
\]
which is a contradiction.
\end{proof}

We extend Definition \ref{isolmin} to the following notion in the case $k=1,2$.

\begin{definition}
A feasible point $\bar x$ is called an isolated local minimizer of order $k$, $k=1,2$ for the problem {\rm (P)} iff there exist a constant $A$ and a neighborhood $N\ni\bar x$ such that for all $x\in S\cap N$ there is 
\[
\lambda^*\in C^*,\quad \lambda^*=(\lambda_1^*,\dots,\lambda_n^*)\ne 0,\quad \sum_{i}(\lambda^*_i)^2=1,
\]
which depend on $x$, with
\[
\lambda^*\cdot f(x)\geq \lambda^*\cdot f(\bar x)+A\|x-\bar x\|^k.
\]
\end{definition}

\begin{theorem}
Let $\bar x$ be a feasible point for the problem {\rm (P)}.
Then the following claims are equivalent:

a) $\bar x$ is an isolated local minimizer of second-order;

b) the following conditions hold for all $u\in\E$:
\begin{equation}\label{36}
\ld 1 F(\bar x;u)\ge 0\quad\textrm{and}\quad \ld 2 F(\bar x;0;u)>0, u\ne 0; 
\end{equation}

c) the following conditions
\begin{equation}\label{28}
\ld 1 F(\bar x;u)\ge 0,\quad\forall u\in\E
\end{equation}
and 
\begin{equation}\label{27}
u\ne 0,\;\ld 1 F(\bar x;u)=0,\quad\Rightarrow\quad
\ld 2 F(\bar x;0;u)>0.
\end{equation}
are satisfied.
\end{theorem}
\begin{proof}
It is obvious that the implication b) $\Rightarrow$ c) holds. 

We prove a) $\Rightarrow$ b).  Let $\bar x$ be an isolated local minimizer of second-order. We prove that Conditions (\ref{36}) hold. Suppose that $u\in\E$ is arbitrary chosen. 
It follows from the definition of a second-order isolated minimizer that there exist a constant $A$ and a neighborhood $N\ni\bar x$ such that for every $x\in S\cap N$ there is $\lambda^*\in C^*$, $\lambda^*\ne 0$, $\sum_{i=1}^n(\lambda_i^*)^2=1$ with
\[
\lambda^*\cdot[f(x)-f(\bar x)]\ge A\|x-\bar x\|^2.
\]
If we choose $\mu^*=0$, then $(\lambda^*,\mu^*)\in\Lambda$. By the definition of the function $F$ we have
\[
F(x)\ge\lambda^*\cdot[f(x)-f(\bar x)]\ge A\|x-\bar x\|^2.
\]
Therefore
\begin{equation}\label{37}
\ld 1 F (\bar x;u)=\liminf_{t\downarrow 0,u^\pr\to u,}\,t^{-1}[F(\bar x+t u^\pr)-F(\bar x)]\ge\liminf_{t\downarrow 0,u^\pr\to u,}\, At\norm{u^\pr}^2=0.
\end{equation}
According to Inequality (\ref{37}) we have   $0\in\lsubd 1 F(\bar x)$. 
It follows from here that
\[
\ld 2 F(\bar x;0;u)=\liminf_{t\downarrow 0,u^\pr\to u,}\,2t^{-2}[F(\bar x+t u^\pr)-F(\bar x)] 
\ge\liminf_{t\downarrow 0,u^\pr\to u,}\, 2\, A\norm{u^\pr}^{2}=2\, A\norm{u}^2>0
\]
for all directions $u$ such that $u\ne 0$.

We prove c) $\Rightarrow$ a). Suppose that conditions {\rm (\ref{28})} and {\rm (\ref{27})}  hold. We prove that $\bar x$ is an isolated local minimizer of second-order.
Assume the contrary that $\bar x$ is not an isolated local minimizer of second-order. Therefore, for every sequence $\{\varepsilon_k\}_{k=1}^\infty$ of positive numbers
converging to zero, there exists a sequence $\{x_k\}$ with $x_k\in S$ such that
\begin{equation}\label{26}
\norm{x_k-\bar x}\le \varepsilon_k,\quad
\lambda\cdot [f(x_k)-f(\bar x)]<\varepsilon_k\norm{x_k-\bar x}^2,\quad\forall \lambda\in C^*. 
\end{equation}
Therefore $x_k\to\bar x$. Denote $t_k=\norm{x_k-\bar x}$,
$d_k=(x_k-\bar x)/t_k$. Passing to a subsequence, we may suppose
that $d_k\to  d$ where $\norm{d}=1$. It follows from the definition of the function $F$ that
\[
F(x_k)=\max\{\lambda\cdot[f(x_k)-f(\bar x)]+\mu\cdot g(x_k)\mid (\lambda,\mu)\in\Lambda\}.
\]
Taking into account (\ref{26}) and $g(x_k)\in -K$ we conclude that
$
F(x_k)\leq \varepsilon_k t^2_k.
$
Hence
\[
\ld 1 F (\bar x;d)\le\liminf_{k\to\infty}\, t^{-1}_k[F(\bar x+t_k d_k)-F(\bar x)]
\le\liminf_{k\to\infty}\varepsilon_k t_k=0.
\]
It follows from (\ref{28}) that $0\in\subd F(\bar x)$ and $\ld 1 F (\bar x;d)=0$. We have
\[
\ld 2 F (\bar x;0;d)\le\liminf_{k\to\infty}\,2 t^{-2}_k F(\bar x+t_k d_k)\le\liminf_{k\to 0}\, 2\varepsilon_k=0,
\]
which is contrary to the assumption (\ref{27}).
\end{proof}

\section{Comparison with some previous results}
\label{s8}

In this section, we review a lot of necessary and sufficient optimality conditions in unconstrained optimization and prove that they follow from Theorems \ref{th1} and \ref{th2} as particular cases.

The following necessary conditions in terms of Dini derivatives are well-known:

\begin{proposition}\label{pr2}
Let $\bar x\in\dom f $ be a local minimizer of the function $f$. Then
\begin{equation}\label{21}
\dini f(\bar x;u)\ge 0\quad\textrm{for all}\quad u\in\E
\end{equation}
and
\begin{equation}\label{22}
\dini f(\bar x;u)=0\quad\Rightarrow\quad\secdini f(\bar x;u)\ge 0,
\end{equation}
\end{proposition}

In the next claim and example we compare the lower Hadamard derivatives with the lower Dini derivatives. If some point $x$ do not satisfy  the necessary conditions in terms of Dini derivatives, then it is not a local minimizer. We prove that in this case Theorem \ref{th1} also detects that $x$ is not a local minimizer. In the example, we show that there exist functions such that Theorem \ref{th1} can reject the point as a possible minimizer, but Proposition \ref{pr2} cannot. 

\begin{proposition}
Let $\f$ be an arbitrary proper extended real function. Suppose that $\bar x\in\dom f$ be a point such that
(\ref{21}) or (\ref{22}) do not hold. Then at least one of the inequalities (\ref{15}) is not satisfied.
\end{proposition}
\begin{proof}
Suppose that there exists a direction $u\in\E$ such that (\ref{21}) fails. Then 
\[
\ld 1 f(\bar x;u)\le\dini f(\bar x;u)<0
\]
and the first inequality of (\ref{15}) also fails. 

Suppose that there exists a direction $u\in\E$ such that (\ref{22}) is not satisfied. Then $\ld 1 f(\bar x;u)\le\dini f(\bar x;u)=0$. If there exists a direction $v\in\E$ with $\ld 1 f(\bar x;v)<0$, then (\ref{15}) fails. Otherwise $\ld 1 f(\bar x;v)\ge 0$ for all $v\in\E$. 
Therefore $\ld 1 f(\bar x;u)=0$, $0\in\lsubd 1 f(\bar x)$ and $\ld 2 f(\bar x;0;u)$ is well defined. We have
\[
\ld 2 f(\bar x;0;u) \le \liminf_{t\downarrow 0}\, 2\, t^{-2} [f(\bar x+t u)-f(\bar x)]=\secdini f(\bar x;u)<0,
\]
because $\dini f(\bar x;u)=0$. Thus (\ref{15}) also fails.
\end{proof}

\begin{example}
Consider the function of two variables
\[
f(x_1,x_2)=\left\{
\begin{array}{ll}
-(x_1^2+x_2^2), & \textrm{if}\quad x_2=x_1^2,\; x_1>0, \\
0 & otherwise.
\end{array}\right.
\]
The point $\bar x=(0,0)$ is not a local minimizer. Easy calculations give that $\ld 1 f(\bar x,u)=0$ for every direction $u\in\R^2$. Therefore $(0,0)\in\lsubd 1 f(\bar x)$. We have $\ld 2 f(\bar x,0,\bar u)=-2$, where $\bar u=(1,0)$ and Theorem \ref{th1} detects that $\bar x$ is not a minimizer. On the other hand $f^\pr_D(\bar x,u)=f^{\pr\pr}_D(\bar x,u)=0$ for every direction $u\in\R^2$ and the lower Dini derivatives cannot detect that $\bar x$ is not a minimizer. 
\end{example}

In the paper \cite{pas08}, the authors introduced the so called $l$-stable functions and generalized some earlier conditions for an isolated local minimum in unconstrained optimization to these functions. We prove that the main result in this paper \cite[Theorem 6]{pas08} is a particular case of Theorem \ref{th2}.


\begin{lemma}[\cite{pas08}]\label{lema2}
Let $f:\E\to\R$ be a continuous function on some neighborhood of the point $x\in\E$ and $l$-stable at $x$. Then $f$ is Lipschitz on a neighborhood of $x$.
\end{lemma}



\begin{proposition}[\cite{pas08}]\label{bp08}
Let $f:\E\to\R$ be continuous on some neighborhood of   $x\in\E$ and let $f$ be $l$-stable at $x$. If 
\begin{equation}\label{20}
\dini f(x;h)=0\quad\textrm{and}\quad \secdini f(x;h)>0,\quad\forall h\ne 0,
\end{equation}
then $x$ is an isolated local minimizer of order $2$ for $f$. 
\end{proposition}
\begin{proof}
Suppose that conditions (\ref{20}) are satisfied. Then it follows from Lemma \ref{lema2} that $f$ is Lipschitz on some neighborhood of $x$. Therefore, the lower Hadamard and Dini derivatives coincide, that is $\ld 1 f(x;h)=\dini f(x;h)=0$ for every $h\in S$. It follows from here that $0\in\lsubd 1 f(x)$ and $\ld 2 f(x;0;h)$ exists. 
Equality (\ref{40}) is satisfied  by Lemma \ref{lema4} 
 It follows from here that
\[
\ld 2 f(x;0;h)=\liminf_{t\downarrow 0,h^\pr\to h}\;2[f(x+t h^\pr)-f(x)]/t^2
\]
\[
=\liminf_{t\downarrow 0,h^\pr\to h}\;2[f(x+th)-f(x)]/t^2=\secdini f(x;h),
\]
which implies that $\ld 2 f(x;0;h)=\secdini f(x;h)>0$ for each $h\in S$. Then, by Theorem \ref{th2}, $x$ is an isolated minimizer of second order.
\end{proof}

The following optimality conditions were derived in \cite{ggr06}. 
\begin{proposition}[\cite{ggr06}]\label{GGR}
Consider a given function $f:\E\to\R$, which belongs to the class {\rm C}$^{1,1}$.

\noindent
(Necessary conditions) Let $x$ be a local minimizer of $f$. Then $\nabla f(x)=0$ and for each $u\in \E$ it holds
$\secdini f(x;u)\ge 0$.

\noindent
(Sufficient Conditions) Let the point $\bar x$ satisfy the following conditions:
\begin{equation}\label{ggr}
\nabla f(x)=0,\quad  \secdini f(x;u)>0,\;\forall\;  u\in\E, u\ne 0.
\end{equation}
Then $x$ is an isolated minimizer of second-order. Conversely, every isolated minimizer of second-order satisfies these conditions. 
\end{proposition}
It follows from Lemma \ref{lema4} that Proposition \ref{GGR} is a corollary of Theorems \ref{th1} and \ref{th2}. The conditions for isolated local minimum are particular case of Proposition \ref{bp08}.

Consider the second-order lower directional derivative 
\[
f^{\pr\pr}_{BP}(x;u,v):=\liminf_{t\downarrow 0}\, t^{-1}[\nabla f(x+tu)(v)-\nabla f(x)(v)].
\]

The following result \cite[Theorem 3.1]{bp04}  is a corollary of the sufficient conditions in Theorem \ref{th2}:

\begin{proposition}[\cite{bp04}]\label{BP}
Consider a given function $f:\E\to\R$, which belongs to the class {\rm C}$^{1,1}$ on some neighborhood of $x$. If
\begin{equation}\label{bp}
\nabla f(x)=0,\quad{\textrm and}\quad  f^{\pr\pr}_{BP}(x;u,u)>0,\;\forall u\in\E\setminus\{0\},
\end{equation}
then $f$ attains a strict local minimum at $x$.
\end{proposition}
\begin{proof}
Suppose that $f$ satisfies Conditions (\ref{bp}). We prove that Conditions (\ref{ggr}) also hold. Denote $\varphi(t)=f(x+tu)$. It follows from here that $\varphi^\pr(t)=\nabla f(x+tu)(u)$ and $\varphi^\pr(0)=0$. Then
\[
f^{\pr\pr}_{BP}(x;u,u):=\liminf_{t\downarrow 0}\, \varphi^\pr(t)/t,\quad{\textrm and}\quad
\secdini f(x;u)=\liminf_{t\downarrow 0}\, 2[\varphi(t)-\varphi(0)]/t^2.
\]
By Conditions (\ref{bp}) we have $\liminf_{t\downarrow 0}\, \varphi^\pr(t)/t>0$. Therefore, there exist $\alpha>0$ and $\delta>0$ such that
\[
\varphi^\pr(t)>\alpha t,\quad\forall t\in(0,\delta).
\]
On the other hand
\[
\varphi(t)-\varphi(0)=\int_0^t \varphi^\pr(s)ds>\int_0^t\alpha s ds=\alpha t^2/2.
\]
Hence, by Proposition \ref{GGR} and Theorems \ref{th2}, $x$ is an isolated minimizer of second-order. Therefore, the function attains an isolated local minimum of second-order at $x$.
\end{proof}

We will not generalize the necessary conditions in \cite{bp04}, because they are obtained in terms of another second-order upper generalized derivative and the author's proof is very short. We should mention that Proposition \ref{BP} is an generalization of the sufficient conditions by Cominetti, Correa \cite[Proposition 5.2]{cc90} and also \cite[Proposition 6.2]{cha94}, \cite[Theorem 5.1 (ii)]{yan96}, \cite[Theorem 4.2 (ii)]{yan99}. Therefore, these results are covered by our sufficient conditions.

\smallskip
In several papers, R. W. Chaney introduced and studied a second-order directional derivative; see, for example, \cite{cha87}. We recall the definition of the derivative of Chaney.

It is called that a sequence $\{x_k\}$, $x_k\in\E$, $x_k\ne x$ converges to a point $x\in\E$ in direction $u\in\E$, $u\ne 0$ iff the sequence $\{(x_k-x)/\norm{x_k-x}\}$ converges to $u$. 

Let $f:\R^n\to\R$ be a locally Lipschitz function. Denote its Clarke generalized gradient at the point $x$ by $\partial f(x)$. Suppose that $u$ is a nonzero vector in $\R^n$. Denote by $\partial_u f(x)$ the set of all vectors $x^*$ such that there exist sequences $\{x_k\}$ and $\{x_k^*\}$ with $x_k^*\in\partial f(x_k)$, $\{x_k\}$ converges to $x$ in direction $u$, and $\{x_k^*\}$ converges to $x^*$. Really $\partial_u f(x)\subset\partial f(x)$.

\begin{definition}[\cite{cha87}]\label{def5}
Let $f:\R^n\to\R$ be a locally Lipschitz function. Suppose that $x\in\R^n$,  $u\in\R^n$, and $x^*\in\partial_u f(x)$. Then the second-order lower derivative of Chaney $\sld f(x;x^*;u)$ at $(x,x^*)$ in direction $u$ is defined to be the infinimum of all numbers
\[
\liminf\,2[f(x_k)-f(x)-x^*(x_k-x)]/t^2_k,
\]
taken over all triples of sequences $\{t_k\}$, $\{x_k\}$, and $\{x_k^*\}$ for which

\noindent {\rm (a)} $t_k>0$ for each $k$ and $\{x_k\}$ converges to $x$, \\
\noindent {\rm (b)} $\{t_k\}$ converges to $0$ and $\{(x_k-x)/t_k\}$ converges to $u$, \\
\noindent {\rm (c)} $\{x_k^*\}$ converges to $x^*$ with $x_k^*\in\partial f(x_k)$ for each $k$.
\end{definition}


The following claims due to Huang and Ng \cite[Theorems 2.2, 2.7 and 2.9]{hua94} are important necessary and sufficient conditions for optimality in unconstrained optimization. The necessary conditions are generalizations of the respective results due to Chaney \cite[Theorem 1]{cha87}, where the function is semismooth.

\begin{proposition}[\cite{hua94}]\label{pr7}
Let  $f:\R^n\to\R$ be a locally Lipschitz function. Suppose that  $\dini f(x;v)\ge 0$, for all $v\in\R^n$.
For $u\in\R^n$ with norm $1$, if $\dini f(x;u)=0$, then $0\in\partial_u f(x)$.
\end{proposition}

\begin{lemma}\label{lema1}
Let $f:\R^n\to\R$ be a locally Lipschitz function. Suppose that $x\in\R^n$ and $u\in\R^n$. If $0\in\partial f_u(x)$ and $0\in\lsubd 1 f(x)$, then $\sld f(x;0;u)=\ld 2 f(x;0;u)$.
\end{lemma}
\begin{proof}
Denote $u_k=(x_k-x)/t_k$. Then
\[
\sld f(x;0;u)=\liminf\,2[f(x+t_k u_k)-f(x)]/t^2_k,
\]
where the limes infinimum is taken over all pairs of sequences $\{t_k\}$, $\{u_k\}$, which satisfy Conditions (a) and (b) from Definition \ref{def5}. It follows from here that
\[
\sld f(x;0;u)=\liminf_{t\downarrow 0,u^\pr\to u}\,2[f(x+t u^\pr)-f(x)]/t^2=\,\ld 2 f(x;0;u),
\]
which completes the proof.
\end{proof}

\begin{proposition}[\cite{hua94}]\label{pr5}
Let  $f:\R^n\to\R$ be a locally Lipschitz function. Suppose that 
\[
\dini f(\bar x;v)\ge 0,\quad\forall v\in\R^n,\; v\ne 0.
\]
If $\sld f(\bar x;0;u)>0$ for all unit vectors $u\in\R^n$ for which $\dini f(\bar x;u)=0$, then $\bar x$ is a strict local minimizer.
\end{proposition}

\begin{proof}
Since $f$ is locally Lipschitz, then $\ld 1 f(\bar x;v)=\dini f(\bar x;v)\ge 0$ for all directions $v\in\R^n$. Therefore $0\in\lsubd 1 f(\bar x)$. Suppose that
$\dini f(\bar x;u)=0$ for some unit direction $u$. It follows from Lemma \ref{lema1} that 
\[
\ld 2 f(\bar x;0;u)=\sld f(\bar x;0;u)>0,
\]
 because by Proposition \ref{pr7} we have $0\in\partial_u f(\bar x)$. Then, according to Theorem \ref{th2} the point $\bar x$ is an isolated local minimizer.
\end{proof}

It is seen that our proof is shorter than the proof in \cite{hua94}.

\begin{proposition}[\cite{hua94}]\label{pr4}
Let $\bar x$ be a local minimum point of the locally Lipschitz function $f$ and $u\in\R^n$ with norm $1$ such that $\dini f(\bar x;u)=0$. Then 
\[
0\in\partial_u f(\bar x)\quad{\textrm and }\quad \sld f(\bar x;0;u)\ge 0.
\]
\end{proposition}

\begin{proof}
Let $\dini f(\bar x;u)=0$ for some unit direction $u$. By Propositions \ref{pr2} and \ref{pr7}, we have $0\in\partial_u f(\bar x)$.
Then, by Lemma \ref{lema1},  $\sld f(\bar x;0;u)=\ld 2 f(\bar x;0;u)$. Thus the claim follows from Theorem \ref{th1}.
\end{proof}

Ben-Tal and Zowe introduced the following second-order derivative of a function $f:\E\to\R$ at the point $x\in\E$ in directions $u\in\E$ and $z\in\E$:
\[
f^{\pr\pr}_{BZ}(x;u,z):=\lim_{t\downarrow 0}\, t^{-2}[f(x+tu+t^2 z)-f(x)-t f^{\pr}(x;u)],
\]
where $f^{\pr}(x;u):=\lim_{t\downarrow 0}\, t^{-1}[f(x+tu)-f(x)]$ is the usual directional derivative of first-order. 

The following conditions are necessary for a local minimum in terms of the derivative of Ben-Tal and Zowe \cite{bt85}:
\begin{proposition}
Let $\bar x$ be a local minimizer of $f:\E\to\R$. Then
\begin{equation}\label{41}
f^\pr(\bar x;u)\ge 0,\quad\forall u\in\E,
\end{equation}
\begin{equation}\label{42}
f^\pr(\bar x;u)=0\quad\Rightarrow\quad f^{\pr\pr}_{BZ}(\bar x;u,z)\ge 0,\;\forall z\in\E.
\end{equation}
\end{proposition}

In the next result, we prove that Conditions (\ref{41}) and (\ref{42}) are consequence of (\ref{15}):
\begin{proposition}
Let $f:\E\to\R$ and $\bar x\in\E$ be a given function and a point respectively,  such that the derivatives $ f^\pr(\bar x;u)$ and $f^{\pr\pr}_{BZ}(\bar x;u,z)$ exist for all directions $u\in\E$ and $z\in\E$. Then Conditions (\ref{15}) imply that $(\ref{41})$ and $(\ref{42})$ are satisfied at $\bar x$.
\end{proposition}
\begin{proof}
Suppose that (\ref{15}) holds. Then the inequality $f^\pr(\bar x;u)\ge\ld 1 f(\bar x;u)\ge 0$ implies that (\ref{41}) is satisfied. Let  $f^\pr(\bar x;u)=0$. Then the chain of relations
\begin{gather*}
f^{\pr\pr}_{BZ}(\bar x;u,z)=\lim_{t\downarrow 0} t^{-2}[f(\bar x+t(u+tz))-f(\bar x)] \\
\ge\liminf_{t\downarrow 0,u^\pr\to u} t^{-2}[f(\bar x+tu^\pr)-f(\bar x)]=
0.5\,\ld 2 f(\bar x;0;u)\ge 0 
\end{gather*}
show that (\ref{42}) is also satisfied for arbitrary $z\in\E$.
\end{proof}

It was obtained by Auslender \cite{aus84} (see Proposition 2.1 and Corollary 2.2) sufficient conditions and necessary ones for an isolated minimum of second-order for a given locally Lipschitz function. The necessary conditions are derived in terms of the upper Dini directional derivative and the sufficient ones in terms of the lower Dini directional derivative. The second-order derivative, which is used in them,  is not consistent with the classical second-order Fr\'echet directional derivative. Both the necessary and the sufficient conditions are consequence of Theorem \ref{th2} taking into account that the function is locally Lipschitz and the Hadamard and Dini derivatives coincide in this case.

Let $f: X\to\R$, be a given C$^{1,1}$ function, defined on an open set $X\subset\R^n$, and $x_0\in X$. The generalized Hessian matrix of $f$ at $x_0$ \cite{husn84}, denoted by $\partial^2 f(x_0)$, is the set of matrices defined as the convex hull of the set

\centerline{\{$M\mid\exists x_i\to x_0$ with $f$ twice differentiable at $x_i$ and $\nabla^2 f(x_i)\to M$\}.}

\noindent
By construction $\partial^2 f(x_0)$ is a nonempty compact convex set. The support bifunction of $\partial^2 f(x_0)$ is the second-order generalized derivative
\[
f^{\pr\pr}_{HUSN}(x;u,v):=\limsup_{x\to x_0,t\downarrow 0}\, t^{-1}[\nabla f(x+tu)(v)-\nabla f(x)(v)].
\]
The following necessary conditions are Theorem 3.1 in \cite{husn84}. We prove that it as a consequence of Theorem \ref{th1}.

\begin{proposition}\label{pr-a}
Let $f$ be a {\rm C}$^{1,1}$ function on some open set $X\subset\R^n$ and $x_0$ be a local minimizer of $f$ over $X$. Then for every direction $d\in\R^n$ there exists a matrix $A\in\partial^2 f(x_0)$ such that $\scalpr{Ad}{d}\ge 0$.
\end{proposition} 
\begin{proof}
Suppose the contrary that there exists $d\in\R^n$ such that $\scalpr{Ad}{d}<0$ for every matrix $A\in\partial^2 f(x_0)$. Since $\partial^2 f(x_0)$ is nonempty and compact, then we have $f^{\pr\pr}_{HUSN}(x_0;d,d)<0$. Therefore 
\[
\limsup_{t\downarrow 0}\, t^{-1}[\nabla f(x_0+td)(d)-\nabla f(x_0)(d)]<0.
\]
Denote $\varphi(t)=f(x_0+td)$. It follows from here that $\varphi^\pr(t)=\nabla f(x_0+td)(d)$ and $\varphi^\pr(0)=0$. Then
\[
\limsup_{t\downarrow 0}\, \varphi^\pr(t)/t<0
\]
Therefore, there exist $\alpha>0$ and $\delta>0$ such that
\[
\varphi^\pr(t)<-\alpha t,\quad\forall t\in(0,\delta).
\]
On the other hand
\[
\varphi(t)-\varphi(0)=\int_0^t \varphi^\pr(s)ds<-\int_0^t\alpha s ds=-\alpha t^2/2.
\]
It follows from here that
\[
\secdini f(x;d)=\liminf_{t\downarrow 0}\, 2[\varphi(t)-\varphi(0)]/t^2<\limsup_{t\downarrow 0}\, 2[\varphi(t)-\varphi(0)]/t^2<-\alpha<0
\]
This result contradicts Proposition \ref{GGR} and Theorem \ref{th1}.
\end{proof}

Sufficient conditions for optimality were not obtained in \cite{husn84}. Therefore we have nothing to compare  with our results.

It is easy to prove the following necessary conditions 
by the arguments of Proposition \ref{pr-a}:

\begin{proposition}\label{pr-b}
Let $f$ be a {\rm C}$^{1,1}$ function on some open set $X\subset\R^n$ and $x_0$ be a local minimizer of $f$ over $X$. Then 
\[
\limsup_{t\downarrow 0}\,t^{-1}[\nabla f(x_0+td)(d)-\nabla f(x_0)(d)]\ge 0,\quad\forall d\in\R^n.
\]
\end{proposition} 

Consider the second-order lower directional derivative \cite{yj92}: 
\[
f^{\pr\pr}_{YJ}(x;u,v):=\sup_{z\in X}\limsup_{t\downarrow 0}\, t^{-1}[\nabla f(x+tz+tu)(v)-\nabla f(x+tz)(v)].
\]

\begin{proposition}[\cite{yj92}]\label{pr-c}
Let $f$ be a {\rm C}$^{1,1}$ function on some open set $X\subset\R^n$ and $x_0$ be a local minimizer of $f$ over $X$. Then 
$f^{\pr\pr}_{YJ}(x;d,d)\ge 0$ for every direction $x\in\R^n$.
\end{proposition} 

Proposition \ref{pr-b} is sharper than Propositions \ref{pr-a}, \ref{pr-c} and the necessary optimality conditions in the paper \cite{bp04}, also Theorem 7.1 in the paper \cite{jl98}, because they obviously follow from Proposition \ref{pr-b}.

In several papers Rockafellar studied the epi-derivatives, which were introduced by the same author. We prove that the optimality conditions for unconstrained problems follow from our results as particular case.

\begin{definition}[\cite{roc88}]
Let a family of subsets $S_t\subset\E$, which is parametrized (or indexed) by $t>0$, be given. One says that $S_t$ converges to a subset $S$ as $t\downarrow 0$, written $S=\lim_{t\downarrow 0} S_t$, iff
\[
S=\limsup_{t\downarrow 0} S_t=\liminf_{t\downarrow 0} S_t,
\]
where $\limsup_{t\downarrow 0} S_t$ and $\liminf_{t\downarrow 0} S_t$ are the upper Kuratowski-Painlev\'e limit of sets and the lower one.
\end{definition}

\begin{definition}[\cite{roc88}]
Consider a family of functions $\varphi:\R^n\to\overline\R$, where $\overline\R=[-\infty,\infty]$. One says that $\varphi_t$ epi-converges to a function $\varphi:\R^n\to\overline\R$ as $t\downarrow 0$, written
\[
\varphi={\rm epi-}\lim_{t\downarrow 0}\varphi_t,
\]
iff the epigraphs {\rm epi}$\,\varphi_t$  converge to the epigraph {\rm epi}$\,\varphi$ in $\R^n\times\R$ as $t\downarrow 0$ in the sense of Kuratowski-Painlev\'e. 
\end{definition}

Note that in this case $\varphi$ must be  a lower semicontinuous function, if it is a epi-limit.
It follows from this definition that
\[
\varphi(\xi)=\limsup_{t\downarrow 0}\inf_{\xi^\pr\to\xi} \varphi_t(\xi^\pr)=\liminf_{t\downarrow 0}\inf_{\xi^\pr\to\xi} \varphi_t(\xi^\pr),
\]
where
\[
\limsup_{t\downarrow 0}\inf_{\xi^\pr\to\xi} \varphi_t(\xi^\pr)=\lim_{\varepsilon\downarrow 0}\lim_{\tau\downarrow 0}
\sup_{t\in (0,\tau)}\inf_{\xi^\pr\in\xi+\varepsilon B} \varphi_t(\xi^\pr),
\]
\begin{equation}\label{rtr}
\liminf_{t\downarrow 0}\inf_{\xi^\pr\to\xi} \varphi_t(\xi^\pr)=\lim_{\varepsilon\downarrow 0}\lim_{\tau\downarrow 0}
\inf_{t\in (0,\tau)}\inf_{\xi^\pr\in\xi+\varepsilon B} \varphi_t(\xi^\pr)=\liminf_{t\downarrow 0,\;\xi^\pr\to\xi}\varphi_t(\xi^\pr).
\end{equation}
Here $B$ is the unit closed ball centered at the origin (see \cite[p. 82]{roc88}).

\begin{definition}[\cite{roc88}]
A function $f$ is said to be epi-differentiable at a point $x$ iff the first-order difference quotient functions
\[
\varphi_{x,t}(\xi)=[f(x+t\xi)-f(x)]/t\quad\textrm{for } t>0
\]
have the property that the limit function $f^\pr_x:={\rm epi-}\lim_{t\downarrow 0}\varphi_{x,t}$ exists and $f^\pr_x(0)>-\infty$. Then the values $f^\pr_x(\xi)$ are called first-order directional derivatives of $f$ at $x$. A vector $v\in\E$ is a epi-gradient of $f$ at $x$ iff
$f^\pr_x(\xi)\ge\scalpr{\xi}{v}$ for all $\xi\in\E.$
\end{definition}

\begin{definition}[\cite{roc88}]
A function $f$ is called twice epi-differentiable at $x$ relative to a vector $v$ iff it is (once) epi-differentiable at $x$ in the sense of the preceding definition and the second-order difference quotient functions
\[
\varphi_{x,v,t}(\xi)=2[f(x+t\xi)-f(x)-t\scalpr{\xi}{v}]/t^2
\]
have the property that the limit function $f^{\pr\pr}_{x,v}:={\rm epi-}\lim_{t\downarrow 0}\varphi_{x,v,t}$
exists and $f^{\pr\pr}_{x,v}(0)>-\infty$. Then the values $f^{\pr\pr}_{x,v}(\xi)$ are called second-order (directional) epi-derivatives of $f$ at $x$ relative to $v$.
\end{definition}

\begin{proposition}[\cite{roc89}]\label{pr3}
Let $f:\R^n\to\overline\R$ be a lower semicontinuous function, and let $x$ be a point where $f$ is finite and twice epi-differentiable

{\rm (a)} (Necessary condition). If $f$ has a local minimum at $x$, then $0$ is an epi-gradient of $f$ at $x$ and $f^{\pr\pr}_{x,0}\ge 0$ for all $\xi$.

{\rm (b)} (Sufficient conditions) If 0 is an epi-gradient of $f$ at $x$ and $f^{\pr\pr}_{x,0}>0$ for all $\xi\ne 0$, then $f$ has a second-order isolated  local minimum at $x$.
\end{proposition}
\begin{proof}
It follows from (\ref{rtr}) that $\ld 1 f(x;\xi)=f^\pr_x(\xi)$, the set of epi-gradients coincides with $\lsubd 1 f(x)$ when the function is epi-differentiable. Again, by (\ref{rtr}), $\ld 2 f(x;v^*;\xi)=f^{\pr\pr}_{x,v^*}(\xi)$ when the function is twice epi-differentiable. Therefore, the necessary conditions follow from Theorem 1, the sufficient conditions follow from Theorem \ref{th2}.
\end{proof}


A similar notion was introduced by Cominetti \cite{com91}, where the epi-convergence is replaced by Mosco convergence of sets. The optimality conditions for unconstrained problems are also particular case of Theorems \ref{th1} and \ref{th2}.

The derivative, which were used in \cite{hua05}, do not coincides with the second-order classical derivative. Therefore, it is not a real derivative. The conditions there cannot be generalized to higher-order ones.

\end{document}